\documentclass[letterpaper, 10 pt, conference]{ieeeconf}  

\IEEEoverridecommandlockouts                              
\usepackage{epstopdf}
\usepackage{graphicx}
\usepackage{amsmath,amssymb,amsfonts}
\usepackage{algorithmic}
\usepackage{textcomp}
\usepackage{amstext}
\usepackage{xcolor,mathdots}
\usepackage{hhline}
\usepackage{multirow}
\usepackage{makecell}
\usepackage{float} 
\usepackage{subfigure}
\usepackage{url}
\usepackage{psfrag}

\newtheorem{proposition}{Proposition}

\title{\LARGE \bf
Multivariable simultaneous stabilization:
A modified Riccati approach
}

\author{Yufang Cui$^{1}$,  and Anders Lindquist$^{2}$
	\thanks{$^{1}$Department of Automation, Shanghai
		Jiao Tong University, Shanghai, China. {\tt\small cui-yufang@sjtu.edu.cn}}%
	\thanks{$^{2}$Department of Automation and School of Mathematical Sciences, Shanghai
		Jiao Tong University, Shanghai, China. {\tt\small alq@math.kth.se}}%
}

\begin{document}

\maketitle
\thispagestyle{empty}
\pagestyle{empty}

\begin{abstract}

Simultaneous stabilization problem arises in various systems and control applications. This paper introduces a new approach to addressing this problem in the  multivariable scenario, building upon our previous findings in the scalar case. The method utilizes a Riccati-type matrix equation known as the Covariance Extension Equation, which yields all solutions parameterized in terms of a matrix polynomial. The procedure is demonstrated through specific examples.

\end{abstract}

\section{Introduction}\label{Introduction}
To present and analyze the issues addressed in this paper, it is necessary to establish the following notation.
\begin{align*}
&\mathbb{C} : \text{the complex plane} \\
&\mathbb{R} : \text{the real line} \\
&\mathbb{C_{-}} : \text{open left half of the complex plane} \\
&\mathbb{C_{+}}: \{ \mathbb{C} - \mathbb{C_{-}} \} \cup \{\infty\} \\
&\mathbb{H} : \text{ring of proper rational functions with real}\\
&\qquad\text{co-efficients with poles in } \mathbb{C_{-}} \\
&\mathbb{H}^{p \times m} : \text{set of } p \times m \text{ matrices whose elements belong to } \mathbb{H} \\
&\mathbb{J} : \text{set of multiplicative units in } \mathbb{H} 
\end{align*}

In the realm of control systems engineering, the challenge of achieving stability in multiple systems through a single controller is a topic of significant interest and practical importance \cite{Blondel,Nett,Chaoui}.  This challenge is encapsulated in the concept of simultaneous stabilization, a pivotal area of study that has garnered considerable attention in both theoretical and applied research \cite{Smith,Wang}. 

Simultaneous stabilization problem involves the identification of a singular controller capable of stabilizing multiple plants  \cite{Vidyasagar,Youla}. In our previous paper \cite{CL_simul}, we addressed the simultaneous stabilization problem for single-input single-output (SISO) systems, which is: Given a family of SISO proper transfer functions denoted as $p_\lambda(s)$ and expressed as:
\begin{equation}\label{systems}
	p_\lambda(s)=\frac{\lambda x_1(s)+(1-\lambda)x_0(s)}{\lambda y_1(s)+(1-\lambda)y_0(s)}
\end{equation}
where $\lambda$ is a parameter in the interval $[0,1]$, and $x_0(s),y_0(s),x_1(s),y_1(s)\in \mathbb{H}$, the objective is to determine a proper compensator $k(s)$ so that the closed-loop systems $p_\lambda(s)(1+k(s)p_\lambda(s))^{-1}$  remain stable for all $\lambda$ in the interval $[0,1]$.

As a generalization of the SISO case, the multi-input and multi-output (MIMO) simultaneous stabilization problem solves the following problem: Suppose there is a family of plants $P_\lambda(s)$ as follows:
\begin{equation}\label{MIMO}
    P_\lambda(s)=N_\lambda(s)D_\lambda(s)^{-1}
\end{equation}
where $\lambda\in[0,1]$ and 
\begin{equation}
    N_\lambda=\lambda N_1+(1-\lambda)N_0
\end{equation}
\begin{equation}
    D_\lambda=\lambda D_1+(1-\lambda)D_0
\end{equation}
with $N_0,N_1,D_0,D_1 \in \mathbb{H}^{m \times m}$, find a proper $m\times m$ compensator $K(s)$ which stabilizes all $P_\lambda(s)$.

In \cite{Ghosh,Ghosh2,Ghosh3}, BK Ghosh concentrated on addressing the simultaneous partial pole placement problem and introduced an interpolation method to solve this problem, offering a fresh perspective for resolving the simultaneous stabilization problem. In this paper, we employ a more comprehensive interpolation approach based on our prior research on a Riccati-type method for analytic interpolation \cite{CLtac, CLcdc}, which is built upon algorithms for the partial stochastic realization problem \cite{b2,BGuL,b1,BLpartial} and on \cite{b6}. Specifically, we convert the simultaneous stabilization problem into a matrix-valued analytic interpolation problem, which can be generally formulated as follows:  Find a real rational Carath\'eodory function $F$ of size $\ell\times \ell$, i.e., a function $F$ that is analytic within the unit disc $\mathbb{D}=\{z\mid |z|<1\}$, and satisfies the inequality condition
\begin{equation}
\label{F+F*}
F(z)+F(z)^{*}  > 0, \quad z\in\mathbb{D},
\end{equation}
and also fulfills the interpolation conditions
\begin{align}
\label{interpolation}
 \frac{1}{j!}F^{(j)}(z_{k})=W_{kj},\quad& k=0,1,\cdots,m,   \\
    &   j=0,\cdots n_{k}-1 ,\notag
\end{align}
where $'$ denotes transposition, $F^{(j)}(z)$ is the $j$th derivative of $F(z)$, and $z_0,z_1,\dots,z_m$ are distinct points in $\mathbb{D}$ and $W_{kj}\in\mathbb{C}^{\ell\times\ell}$ for each $(k,j)$. The complexity of the rational function $F(z)$ is constrained by limiting its McMillan degree to be at most $\ell n$, where
\begin{equation}
\label{deg(f)}
n=\sum_{k=0}^{m}n_k -1 .
\end{equation}

This problem has an infinite number of solutions. However, as we will demonstrate in Section~\ref{sec:CEE}, the freedom to select the $n\times n$ parameter $\Sigma$ enables us to adjust the solution to specific requirements.

The paper is organized as follows: Section~\ref{sec:3condition} outlines the essential conditions necessary for a group of plants to be simultaneously stabilizable and illustrates how the simultaneous stabilization problem can be converted into an analytic interpolation problem. Section~\ref{sec:CEE} delves into solving the analytic interpolation problem using the Covariance Extension Equation (CEE). In Section~\ref{sec:applications}, we present simulations to illustrate how the method can be utilized to stabilize multiple plants. Finally, in Section~\ref{sec:conclusion}, we provide concluding remarks and recommendations for future research.

\section{The simultaneous stabilization problem}\label{sec:3condition}

As explained in \cite{Ghosh,Ghosh2,Ghosh3}, every SISO system can be written as $x(s) / y(s)$, where $x(s), y(s) \in \mathbb{H}$, and an $m \times m$  plant $P(s)$ has the left coprime representation  $D_l(s)^{-1} N_l(s)$ and the right coprime representation $N_r(s)D_r(s)^{-1}$, where $N_l(s),D_l(s),N_r(s),D_r(s) \in \mathbb{H}^{m \times m}$.

To solve the simultaneous stabilization problem, we firstly consider a simple case: Given two different plants 
\begin{equation}
  P_i(s)= N_i(s)D_i(s)^{-1}, \quad i=0, 1  
\end{equation}
where $N_i(s) \in \mathbb{H}^{m \times m}, D_i(s) \in \mathbb{H}^{m \times m}$, for $i=0,1$, find a proper $m\times m$ compensator $K(s)$ which can stabilize $P_0(s)$ and $P_1(s)$.

Set
\begin{equation}
  M(s)=
    \begin{bmatrix}
	N_0(s) & N_1(s) \\
	D_0(s) & D_1(s)
\end{bmatrix},
\end{equation}
and let $\text{Adj}(M(s))$ be its adjoint matrix. Moreover, suppose $\det(M(s))$ has simple zeros in $\mathbb{C_+}$ at $s_1,\cdots,s_t$ and $\det(M(\infty))\neq 0$.
\begin{proposition}
	The two plants $P_0,P_1$ can be simultaneously stabilized by a proper compensator if and only if there exists $\Delta_i(s) \in \mathbb{H}^{m \times m}$, $\operatorname{det} \Delta_i(s) \in \mathbb{J}, i=0, 1$ , such that if $s_{1},s_2$, $\cdots, s_{t}$ are the zeros of $\det(M)$ in $\mathbb{C_+}$, then 
	\begin{equation}
	    \begin{bmatrix}
		\Delta_0(s)&\Delta_1(s)
	\end{bmatrix}\text{Adj}(M(s))=\mathbf{0}
	\end{equation}
	at $s_{1},s_2$, $\cdots, s_{t}$.	
\end{proposition}
\begin{proof}
The proof of this proposition is based on the work of BK Ghosh \cite{Ghosh3}. Let us represent the compensator as
	\begin{equation}
	    K(s)=D_c(s)^{-1}N_c(s) 
	\end{equation}
	where $N_c(s) \in \mathbb{H}^{m \times m}, D_c(s) \in \mathbb{H}^{m \times m}, N_c(s), D_c(s)$ are coprime. Then the compensator stabilizes $P_0(s),P_1(s)$ if and only if
 \begin{equation}
     \quad N_c(s)N_i(s) +D_c(s)D_i(s) =\Delta_i(s)
 \end{equation}
	for some $\Delta_i(s) \in \mathbb{H}^{m \times m}$, $\operatorname{det} \Delta_i(s)\in \mathbb{J}, i=0, 1$ respectively \cite{Ghosh3}. 
	We can write this in matrix form as
	\begin{equation}
	    \begin{bmatrix}
	        N_c&D_c
	    \end{bmatrix}
		\begin{bmatrix}
		    N_0(s) & N_1(s) \\
		D_0(s) & D_1(s)
		\end{bmatrix}
	=\begin{bmatrix}
		\Delta_0(s) ~\Delta_1(s)
	\end{bmatrix},
	\end{equation}
   and  then
	 \begin{equation}
	     \begin{bmatrix}
	        N_c&D_c
	    \end{bmatrix}=\begin{bmatrix}
		\Delta_0(s) ~\Delta_1(s)
	\end{bmatrix}\frac{\text{Adj}(M(s))}{\det(M(s))}.
	 \end{equation}
	 In order to ensure that $N_c(s) \in \mathbb{H}^{m \times m}, D_c(s) \in \mathbb{H}^{m \times m}$, it is necessary and sufficient that 
	\begin{equation}
		\label{matrx_cond}
			[\begin{matrix}
			\Delta_0(s)~\Delta_1(s)
		\end{matrix}]\text{Adj}(M(s))=\mathbf{0}		
	\end{equation}
at $s_{1},s_2$, $\cdots, s_{t}$, where $s_{1},s_2$, $\cdots, s_{t}$  are the simple zeros of $\det(M)$ in $\mathbb{C_+}$.
\end{proof}

Assume generically that $\text{Adj}(M(s_i))$ are of rank 1 at $s_{1},s_2$, $\cdots, s_{t}$, If $\text{Adj}(M(s_i))$ are spanned by a column vector $\mathbf{r_i}=[\mathbf{r_{i1}},\mathbf{r_{i2}}]'$, then
\begin{equation}\label{delta}
    \Delta_0(s_i)\mathbf{r_{i1}}'+\Delta_1(s_i)\mathbf{r_{i2}}'=\mathbf{0}
\end{equation}
for $i=0, \cdots, t$. From equation \eqref{delta}, we can derive $(\Delta_0(s_i))^{-1}\Delta_1(s_i)$.

Now we consider the more general case, suppose
\begin{equation}
    P_\lambda(s)=N_\lambda(s)D_\lambda(s)^{-1}
\end{equation}
where $\lambda\in[0,1]$ and 
\begin{equation}
    N_\lambda=\lambda N_1+(1-\lambda)N_0
\end{equation}
\begin{equation}
    D_\lambda=\lambda D_1+(1-\lambda)D_0
\end{equation}

\begin{proposition}

If $(\Delta_0)^{-1}\Delta_1$ is diagonalizable at $s\in \mathbb{C}_+$, then the set of plants $P_\lambda(s)$ for $\lambda\in[0,1]$ can be simultaneously stabilized by a proper compensator if and only if there exist $\Delta_i(s) \in \mathbb{H}^{m \times m}$, with $\operatorname{det} \Delta_i(s) \in \mathbb{J}$ for $i=0, 1$, satisfying the condition of Proposition 1, along with the additional condition that the eigenvalues of $\Delta_0(s)^{-1}\Delta_1(s)$ do not intersect the nonpositive real axis including infinity  at any point in $\mathbb{C}_+$.

\end{proposition}

\begin{proof}
    Suppose $K(s)$ is the required compensator. To stabilize the plants $P_0(s),P_1(s)$ simultaneously, a necessary and sufficient condition is given by Proposition 1. Additionally, $K(s)$ simultaneously stabilizes every other plants $P_\lambda(s)$ for $\lambda \in (0,1) $ if and only if there exist $\Delta_\lambda(s) \in \mathbb{H}^{m \times m}$, $\operatorname{det} \Delta_\lambda(s) \in \mathbb{J}$ such that
    \begin{equation}
        \quad N_c(s)N_\lambda(s) +D_c(s)D_\lambda(s) =\Delta_\lambda(s), \lambda\in (0,1).
    \end{equation}
    By calculation,
    \begin{equation}
        \lambda\Delta_{1}(s)+(1-\lambda)\Delta_{0}(s)=\Delta_{\lambda}(s).
    \end{equation}

Since $\Delta_0(s) \in \mathbb{H}^{m \times m}$ and $\Delta_1(s) \in \mathbb{H}^{m \times m}$, we have  $\Delta_\lambda(s) \in \mathbb{H}^{m \times m},\lambda\in(0,1)$. In order to ensure $\operatorname{det} \Delta_\lambda(s) \in \mathbb{J}$, we need 
\begin{equation}\label{eigenvalues}
    \operatorname{det} (\lambda\Delta_{1}(s)+(1-\lambda)\Delta_{0}(s))\neq 0, \lambda\in(0,1).
\end{equation}

Note that over the complex numbers $\mathbb{C}$, almost every matrix is diagonalizable. Here we suppose $(\Delta_0)^{-1}\Delta_1$ is diagonalizable at $s\in \mathbb{C}_+$ and suppose $\alpha_1(s),\alpha_2(s),\cdots,\alpha_m(s)$ are its eigenvalues, then  \eqref{eigenvalues} means $\alpha_i(s)\neq (1-1/\lambda)$ for all $\lambda\in(0,1)$
at $s_i\in \mathbb{C}_+$. Also due to $\operatorname{det}(\Delta_i)\in\mathbb{J}$, we can conclude that to ensure simultaneous stabilization, the eigenvalues of $\Delta_0(s)^{-1}\Delta_1(s)$ can not  intersect the nonpositive real axis including infinity  at $s_i\in \mathbb{C}_+$. 
\end{proof}

Next, we reformulate the MIMO simultaneous stabilization problem as multivariable analytic interpolation problem.

From Proposition 1, we can get the interpolation constraints
\begin{equation}
    \Delta_0(s_i)^{-1}\Delta_1(s_i)=M_i,\quad i=1,\cdots,t.
\end{equation}
Let 
\begin{equation}
   F_1(s)=(\Delta_0(s)^{-1}\Delta_1(s))^{1/2}
\end{equation}
be the square-root of $\Delta_0(s)^{-1}\Delta_1(s)$,  which means $F_1(s)$ need satisfy
\begin{equation}
    F_1(s_i)=M_i^{1/2},\quad i=0,\cdots,t.
\end{equation}
Using the M{\"o}bius transformation defined by
\begin{equation}\label{Mobius}
    z=\frac{1-s}{1+s}.
\end{equation}
which can map $\mathbb{C_+}$ into the interior of the unit disc, we then set
\begin{equation}
	F(z):=F_1((1-z)(1+z)^{-1})
\end{equation}
which is obviously analytic in $\mathbb{D}$. The following proposition is straightforward.
\begin{proposition}
    If 
    \begin{equation}
        M_i^{1/2}+(M_i^{1/2})^{*}>0,i=0,\cdots,t,
    \end{equation}
    the simultaneous stabilization problem \eqref{MIMO} is simplified to identifying a Carath\'eodory function $F(z)$ such that  the interpolation constraints 
    \begin{equation}
        F(z_i)=M_i^{1/2},i=0,\cdots,t
    \end{equation}
    is satisfied, where $z_i=(1-s_i)/(1+s_i)$. This is a matrix case analytic interpolation problem \eqref{interpolation} when $m=t,n_0=n_1=\cdots=n_m=0$. 
\end{proposition}

 After obtaining the solution for $F(z)$, we can apply the following transformations to obtain the compensator $K(s)$.
Denote
\begin{equation}
    \frac{Adj(M(s))}{\det(M(s))}=\begin{bmatrix}
        m_{11}(s) & m_{12}(s) \\
	m_{21}(s) & m_{22}(s)
    \end{bmatrix},
\end{equation}
then
\begin{equation}
    F_1(s)=F((1-s)(1+s)^{-1})
\end{equation}
\begin{equation}
    K(s)=D_c^{-1}N_c=(m_{12}+F_1^2m_{22})^{-1}(m_{11}+F_1^2m_{21})
\end{equation}

\section{The analytic interpolation problem}\label{sec:CEE}
This section illustrates the approach to address the analytic interpolation problem \eqref{interpolation} employing the Covariance Extension Equation \cite{CLccdc,CLtac,CLcdc}. Without loss of generality, we set $z_{0}=0$ and $W_{0}=\frac{1}{2}I$. Consequently, $F(z)$ can be expressed as:
 \begin{equation}
\label{ }
F(z)=\tfrac12 I + zH(I-zF)^{-1}G,
\end{equation}
where matrices $H\in\mathbb{R}^{\ell\times\ell n}$, $F\in\mathbb{R}^{\ell n\times\ell n}$, $G\in\mathbb{R}^{\ell n\times\ell}$, and all eigenvalues of matrix $F$ reside in $\mathbb{D}$, $(H,F)$ forms an observable pair.

By defining $\Phi_+(z):=F(z^{-1})$, we obtain:
\begin{equation}
\label{ }
\Phi_+(z)=\tfrac12 I + H(zI-F)^{-1}G,
\end{equation}
which has its poles within the unit disc $\mathbb{D}$. Based on \eqref{F+F*}, it follows that:
\begin{equation}
\Phi_+(e^{i\theta})+\Phi_+(e^{-i\theta})'>0, \quad -\pi\leq \theta\leq\pi ,
\end{equation}
and thus $\Phi_+(z)$ is positive real \cite[Chapter 6]{LPbook}. The problem is then reduced to identifying a rational positive real function $\Phi_+(z)$ of degree at most $\ell n$ that meets the interpolation constraints \eqref{interpolation}. By a coordinate transformation $(H,F,G)\to(HT^{-1},TFT^{-1},TG)$, we can select $(H,F)$ in the observer canonical form
\begin{equation}
    H=\text{diag}(h_{t_1},h_{t_2},\dots,h_{t_\ell}) \in \mathbb{R}^{\ell\times n\ell}
\end{equation}
with $h_\nu:=(1,0,\dots,0)\in\mathbb{R}^\nu$, and
\begin{equation}
\label{F}
F=J-AH \in\mathbb{R}^{n\ell\times n\ell}
\end{equation}
where $A\in\mathbb{R}^{n\ell\times \ell}$, $J:=\text{diag}(J_{t_1},J_{t_2},\dots, J_{t_\ell})$ with $J_\nu$ the $\nu\times\nu$ shift matrix
\begin{equation}
J_\nu =\begin{bmatrix}0&1&0&\dots&0\\0&0&1&\dots&0\\\vdots&\vdots&\vdots&\ddots&0\\
0&0&0&\dots&1\\0&0&0&\dots&0\end{bmatrix}.
\end{equation}
Here $t_1,t_2,\dots,t_\ell$ are the {\em observability indices\/} of $\Phi_+(z)$, and 
\begin{equation}
\label{tsum}
t_1+t_2+\dots+t_\ell=n\ell.
\end{equation}
Next define $\Pi(z):=\text{diag}(\pi_{t_1}(z),\pi_{t_2}(z),\dots,\pi_{t_\ell}(z))$, where $\pi_\nu(z)=(z^{\nu-1},\dots,z,1)$, and the $\ell\times\ell$ matrix polynomial
\begin{equation}
\label{A(z)}
A(z)=D(z) +\Pi(z)A,
\end{equation}
where 
\begin{equation}
\label{D(z)}
D(z):=\text{diag}(z^{t_1},z^{t_2},\dots, z^{t_\ell}).
\end{equation}

From Lemma 1 in \cite{CLcdc},
\begin{equation}\label{lemma}
    H(zI-F)^{-1}=A(z)^{-1}\Pi(z),
\end{equation}
and consequently
 \begin{equation}
\label{AinvB}
\Phi_+(z)=\tfrac12 A(z)^{-1}B(z),
\end{equation}
where
\begin{displaymath}
B(z)=D(z) +\Pi(z)B
\end{displaymath}
with
\begin{equation}
\label{AG2B}
B=A+2G.
\end{equation}

Let $V(z)$ denote the minimum-phase spectral factor of
\begin{equation}
V(z)V(z^{-1})'=\Phi(z) := \Phi_+(z) + \Phi_+(z^{-1})' .
\end{equation}
From \cite[Chapter 6]{LPbook}, $V(z)$ has a realization of the form
\begin{equation}
V(z)=H(zI-F)^{-1}K + R,
\end{equation}
and following \eqref{lemma}, it can be expressed as
\begin{equation}
\label{ }
V(z)=A(z)^{-1}\Sigma(z)R,
\end{equation}
where
\begin{equation}
\label{Sigma(z)}
\Sigma(z)=D(z)+\Pi(z)\Sigma 
\end{equation}
with 
\begin{equation}
\label{Sigma}
\Sigma = A+KR^{-1}. 
\end{equation}

From stochastic realization theory \cite[Chapter 6]{LPbook} we obtain 
\begin{align}
  K  & =(G-FPH')(R')^{-1}  \label{K}\\
  RR'  &  = I-HPH' \label{R}
\end{align}
where $P$ is the unique minimum solution of the algebraic Riccati equation
\begin{equation}
\label{Riccati}
P=FPF' + (G-FPH')(I-HPH')^{-1}(G-FPH')'.
\end{equation}
From \eqref{F}, \eqref{Sigma}, \eqref{K} and \eqref{R}, \eqref{Riccati} can be written
\begin{equation}
\label{AREmod}
P=\Gamma (P-PH'HP)\Gamma' +GG' .
\end{equation}
where
\begin{equation}
\label{Gamma}
\Gamma=J-\Sigma H.
\end{equation}

The article \cite{CLtac} demonstrates that $G$ can be expressed as $u + U(\Sigma + \Gamma PH')$, with $u$ and $U$ being fully determined by the interpolation data \eqref{interpolation}. The  analytic interpolation problem involves determining the values of $(A,B)$ based on the given interpolation data \eqref{interpolation} and a specific matrix polynomial $\Sigma(z)$.   

In \cite{CLtac} we conclude that the conditions for the existence of solutions to this problem  only rely on the interpolation data. If the solution exists, it is also shown in \cite{CLtac}  that the {\em Covariance Extension Equation (CEE)}
\begin{subequations}\label{PgCCE}
	\begin{equation} \label{CEE}
		P=\Gamma (P-PH'HP)\Gamma' +G(P)G(P)' 
	\end{equation}
	 with
	\begin{equation}\label{g(P)} 
		G(P)= u +U(\Sigma + \Gamma PH') ,
	\end{equation}
\end{subequations}
 has a unique symmeric solution $P\geq 0$ with the property that $HPH'<1$. Additionally, for every $\Sigma$, there exists a unique solution to the analytic interpolation problem, which is expressed as follows:

\begin{subequations}\label{Psigma2ab}
	\begin{equation}\label{a}
		A=(I-U)(\Gamma PH'+\Sigma)-u
	\end{equation}
	\begin{equation}\label{b}
		B=(I+U)(\Gamma PH'+\Sigma)+u
	\end{equation}
	\begin{equation}\label{rho}
		R=(I-HPH')^{\frac{1}{2}} .
	\end{equation}
\end{subequations}
 The equation \eqref{PgCCE} can be solved using a homotopy continuation approach, as described in \cite{CLtac}. 
 Distinct selections of matrix polynomial $\Sigma(z)$ yield different feasible solutions.

\section{ Computational examples}\label{sec:applications}

\subsection{Example 1}
First we consider a simple MIMO simultaneous stabilization problem with
\begin{equation}
\begin{split}
     N_0&=\left[\begin{matrix}
	1&2\\
	3&1
\end{matrix}\right]\quad D_0=\left[\begin{matrix}
\frac{s-2}{s+6}&1\\
3&\frac{s-2.7}{s+10}
\end{matrix}\right]\\
N_1&=\left[\begin{matrix}
	1&2\\
	4&3
\end{matrix}\right]\quad D_1=\left[\begin{matrix}
	\frac{s-3.2}{s+2.2}&1\\
	3&\frac{s-7.7}{s+1}
\end{matrix}\right]
\end{split}   
\end{equation}

There exist poles that are not stable when the parameter $\lambda$ changes within the range $[0,1]$. In order to enhance the visibility of the poles, we utilize the transformation defined by equation \eqref{trans}, 
\begin{equation}\label{trans}
	z=\frac{1+s}{1-s}.
\end{equation}
This transformation effectively maps left half plane to the interior of the unit circle, while also converting the right half plane to the exterior of the unit circle. A stable system is characterized by having all poles located within the unit circle. 
Fig.~\ref{beforestabilization1} displays the whole set of poles of $P_\lambda$ as $\lambda$ ranges from 0 to 1 in increments of 0.1. 
\begin{figure}[htp]
	\centering
	\includegraphics[width=1\linewidth]{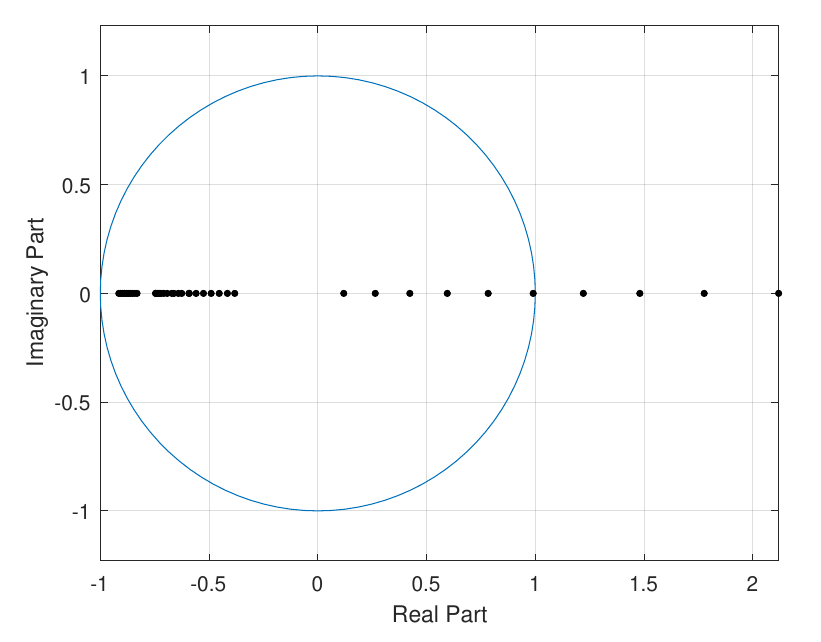}
	\caption{The poles of $P_\lambda$ before stabilization}
	\label{beforestabilization1}
\end{figure}

Based on the information provided in  Fig.~\ref{beforestabilization1}, it is evident that some systems lack stability.  Applying the technique described in this paper, we first observe that $\det(M(s))$ has two zeros at $s_0=12.24$ and $s_1=1.494$ in $\mathbb{C_+}$. In order to achieve system stability, it is necessary to satisfy the interpolation criteria \eqref{delta}, which result in the following equation: 
\begin{equation}
    (\Delta_0(s_0))^{-1}\Delta_1(s_0)=M_0,~(\Delta_0(s_1))^{-1}\Delta_1(s_1)=M_1.
\end{equation}

 Here $M_0+M_0^*>0, M_1+M_1^*>0$, and by using the M{\"o}bius transformation \eqref{Mobius}, the simultaneous stabilization problem then becomes an analytic interpolation problem with $n_0=n_1=1$. The interpolation constraints are
\begin{equation}
\begin{split}
    (\Delta_0^{-1}\Delta_1)(\frac{1-s_0}{1+s_0})&=M_0\\
    (\Delta_0^{-1}\Delta_1)(\frac{1-s_1}{1+s_1})&=M_1.
\end{split}
\end{equation}
According to the theory in \cite{CLtac}, solutions exist in this example. By choosing $\Sigma=[0.3~0;0~0.5]$, we get
\begin{equation}
    (\Delta_0(s))^{-1}\Delta_1(s)=\begin{bmatrix}
        1 & \frac{0.4571 s^2 + 38.42 s + 572.4}{s^2 + 35.22 s + 279.8} \\
        -1 & \frac{0.1414 s^2 + 38.5 s + 869.2}{s^2 + 35.22 s + 279.8} \\
    \end{bmatrix}.
\end{equation}
After stabilization, the locations of the poles of $P_\lambda,\lambda\in[0,1]$ are depicted in Figure~\ref{afterstabilization1}. As all the poles are situated within the open unit disc, it indicates that all feedback systems are stable.

\begin{figure}[htp]
	\centering
	\includegraphics[width=1\linewidth]{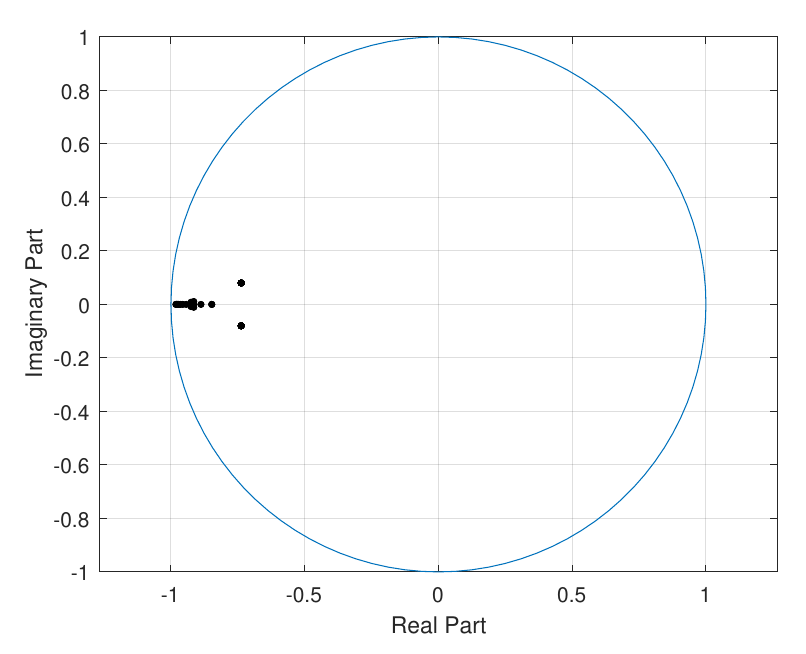}
	\caption{The poles of $P_\lambda$ after stabilization}
	\label{afterstabilization1}
\end{figure}

To demonstrate that different choices of $\Sigma$ produce different feasible solutions, we take $\Sigma$ to be
\begin{equation}
    \begin{split}
        a &= \begin{bmatrix} -0.1 & -0.9 \\ 0.4 & -0.6 \end{bmatrix}\qquad
b = \begin{bmatrix} 0.4 & 0.1 \\ 0.5 & 0.4 \end{bmatrix} \\
c &= \begin{bmatrix} 0.2 & 0.35 \\ 0.6 & 0.4 \end{bmatrix} ~~~\qquad
d = \begin{bmatrix} -0.8 & 0.1 \\ 0.6 & -0.2 \end{bmatrix} \\
e &= \begin{bmatrix} -0.65 & 0.22 \\ 0.8 & -0.2 \end{bmatrix} \qquad
f = \begin{bmatrix} 0.8 & -0.33 \\ 0.9 & 0.7 \end{bmatrix}.
    \end{split}
\end{equation}
respectively. Fig.~3 shows the corresponding results, indicating that the solution changes with different $\Sigma$.
\begin{figure}[htp]
	\centering
	\subfigure[]{
		\begin{minipage}[t]{0.5\linewidth}
			\centering
			\includegraphics[width=1\linewidth]{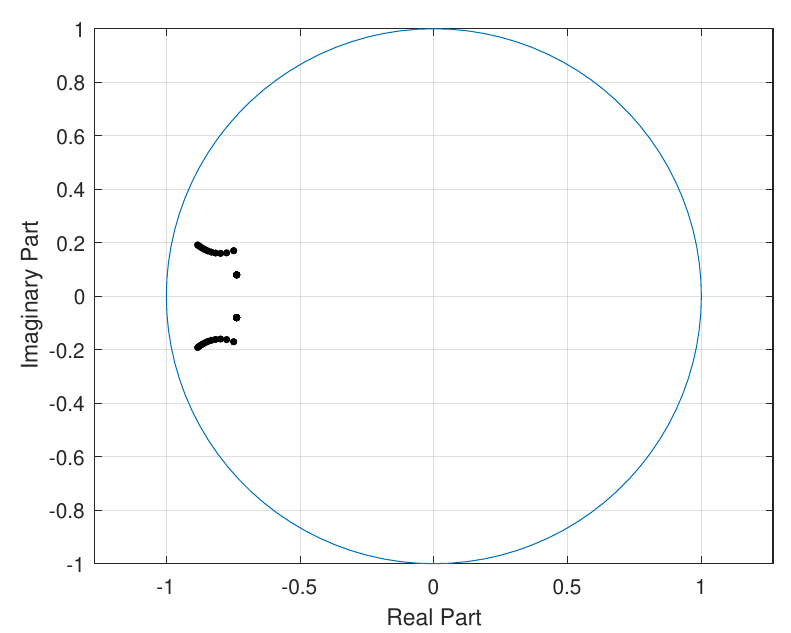}
		\end{minipage}%
	}%
	\subfigure[]{
		\begin{minipage}[t]{0.5\linewidth}
			\centering
			\includegraphics[width=1\linewidth]{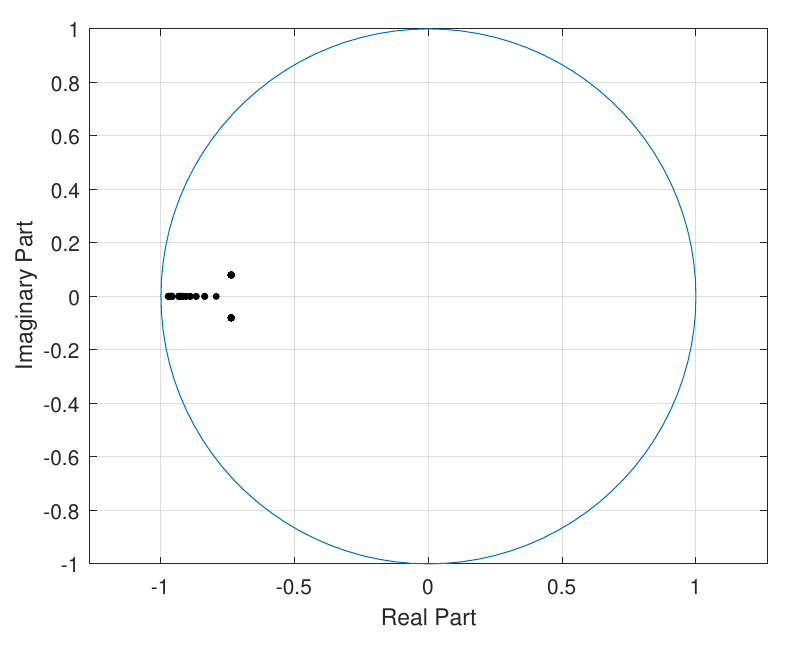}
		\end{minipage}%
	}%

	\subfigure[]{
		\begin{minipage}[t]{0.5\linewidth}
			\centering
			\includegraphics[width=1\linewidth]{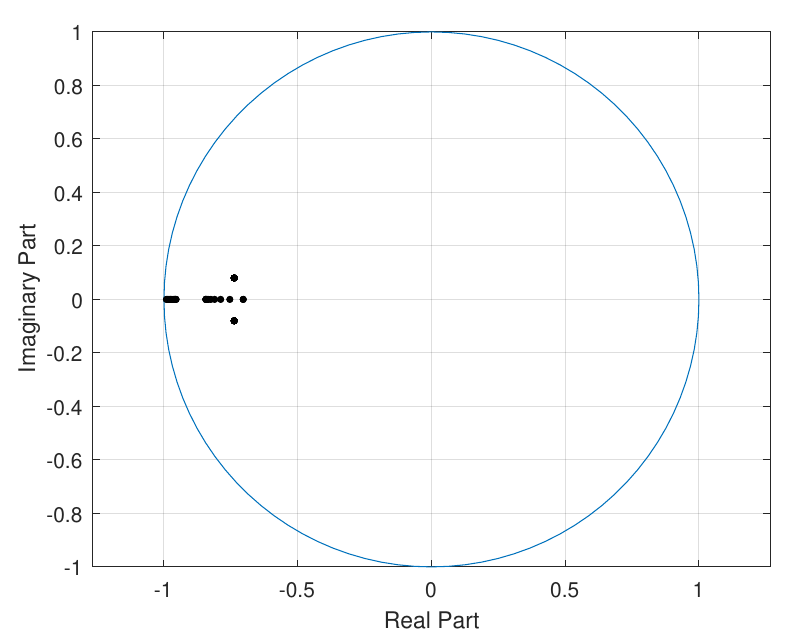}
		\end{minipage}
	}%
	\subfigure[]{
		\begin{minipage}[t]{0.5\linewidth}
			\centering
			\includegraphics[width=1\linewidth]{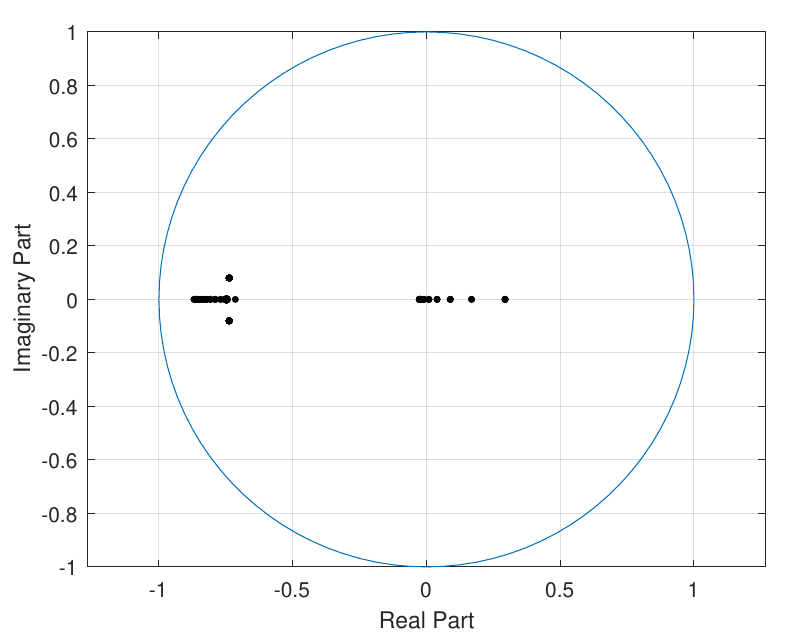}
		\end{minipage}
	}

		\subfigure[]{
		\begin{minipage}[t]{0.5\linewidth}
			\centering
			\includegraphics[width=1\linewidth]{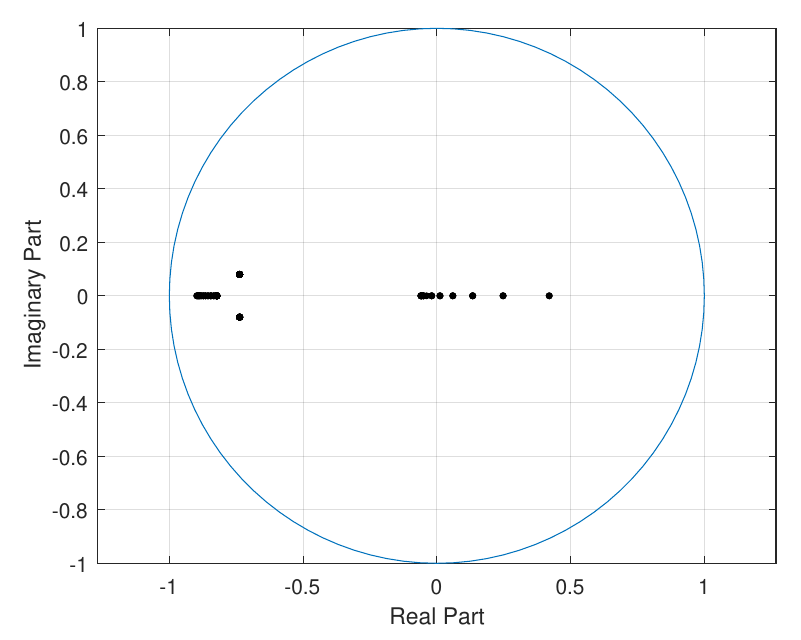}
		\end{minipage}
	}\subfigure[]{
		\begin{minipage}[t]{0.5\linewidth}
			\centering
			\includegraphics[width=1\linewidth]{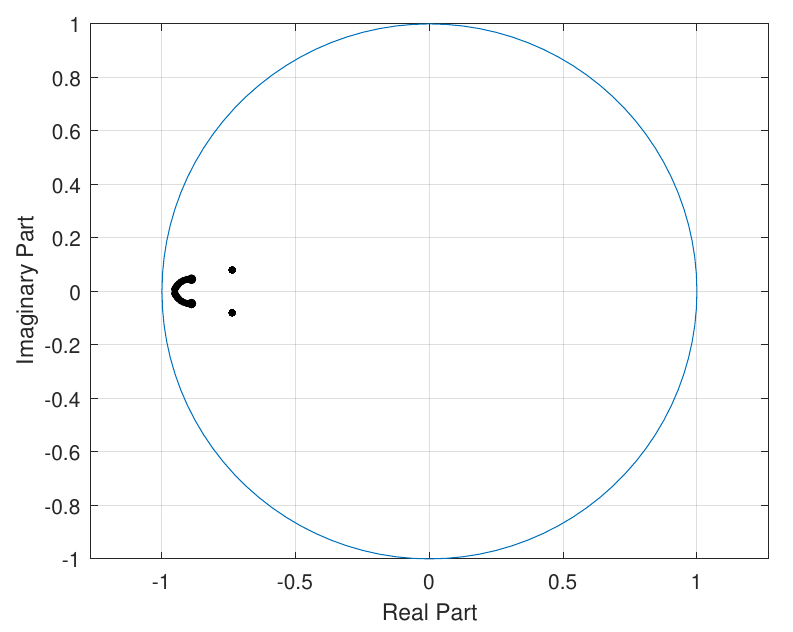}
		\end{minipage}
	}%
	\centering
	\label{sigma_zero}
	\caption{The poles of the stabilized system with diferent $\Sigma$}
\end{figure}
\subsection{Example 2}
Next we consider a more complex system which includes complex unstable zeros. Suppose
\begin{equation*}
   N_0=\frac{1}{z^2+2z+10}\left[\begin{matrix}
	3(z^2-2z+1)&5(z^2+2z+10)\\
	4(z^2+2z+10)&2(z-4)(z-2)
\end{matrix}\right] 
\end{equation*}
\begin{equation*}
    D_0=\frac{1}{z^2+2z+10}\left[\begin{matrix}
2(z-2)(z-3)&-(z^2+2z+10)\\
3(z^2+2z+10)&(z-2)^2
\end{matrix}\right]
\end{equation*}
and
$$
N_1=\frac{1}{z^2+2z+15}\left[\begin{matrix}
	(z-1)(z+1)&6(z^2+2z+15)\\
	7(z^2+2z+15)&(z-8)(z-1)
\end{matrix}\right]
$$
$$
D_1=\frac{1}{z^2+2z+15}\left[\begin{matrix}
	(z-6)(z+2)&2(z^2+2z+15)\\
	-(z^2+2z+15)&3(z-5)(z-9)
\end{matrix}\right]
$$

There are plenty of poles that are unstable as $\lambda$ changes within the range of 0 to 1. In Fig.~\ref{beforestabilization} we show the poles  of $P_\lambda$  as $\lambda$ varies from 0 to 1 at intervals of 0.1. Unlike the previous example, there are many complex poles in this example.

\begin{figure}[htp]
	\centering
	\includegraphics[width=1\linewidth]{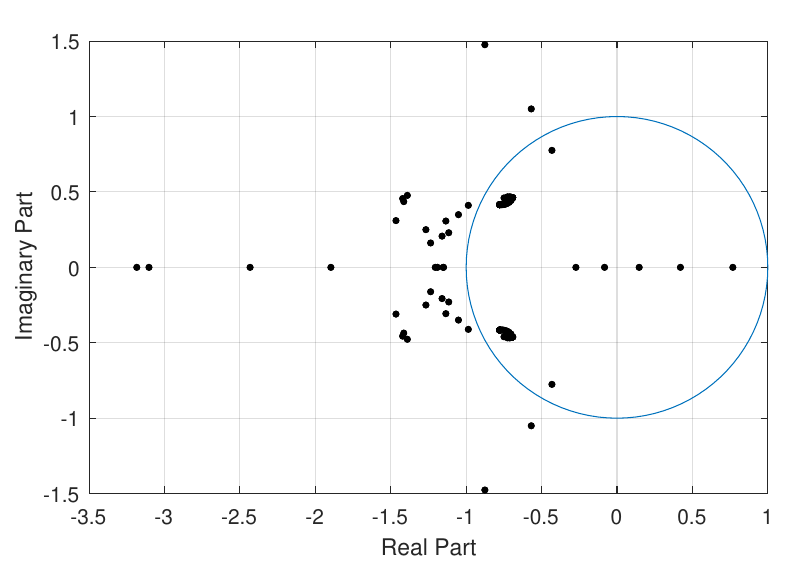}
	\caption{The poles of $P_\lambda$ before stabilization}
	\label{beforestabilization}
\end{figure}

By employing the approach outlined in this paper, we can observe that $\det(M(s))$ has three zeros at $s_0=0.2322$, $s_1=0.9862-3.5291i$, $s_2=0.9862+3.5291i$  in $\mathbb{C_+}$. In order to ensure the stability of the systems, we then need the interpolation conditions \eqref{delta}, which yield
\begin{equation}
    \begin{split}
        (\Delta_0(s_0))^{-1}&\Delta_1(s_0)=M_0,\\
        (\Delta_0(s_1))^{-1}\Delta_1(s_1)=M_1, &
(\Delta_0(s_2))^{-1}\Delta_1(s_2)=M_2,
    \end{split}
\end{equation}
using the M{\"o}bius transformation \eqref{Mobius}, and  since the  Hermitian  parts of $M_0^{1/2},M_1^{1/2},M_2^{1/2}$ are positive,
the problem is then reduced to  the analytic interpolation problem with the  interpolation conditions 
\begin{equation}
\begin{split}
    (\Delta_0^{-1}\Delta_1)^{1/2}(\frac{1-s_0}{1+s_0})&=(M_0)^{1/2}\\
    (\Delta_0^{-1}\Delta_1)^{1/2}(\frac{1-s_1}{1+s_1})&=(M_1)^{1/2}\\(\Delta_0^{-1}\Delta_1)^{1/2}(\frac{1-s_2}{1+s_2})&=(M_2)^{1/2}.
\end{split}
\end{equation}

Here we choose 
\begin{equation}
\Sigma = \begin{bmatrix} 0.4 & 0.2 \\ 0.3 & -0.5 \\ 0.8 & -0.1 \\ 0.6 & -0.2 \end{bmatrix},
\end{equation}
then we get the result
\begin{equation}
    (\Delta_0(s)^{-1}\Delta_1(s))^{1/2}=\begin{bmatrix}
    K_{11}&K_{12}\\
    K_{21}&K_{22}
\end{bmatrix}
\end{equation}
$$K_{11}=\frac{s^4 + 0.923 s^3 + 0.6752 s^2 + 0.1567 s + 0.02193}{s^4 + 0.3846 s^3 + 0.6211 s^2 + 0.1852 s + 0.0186}$$
$$K_{12}=\frac{-0.325 s^4 + 1.821 s^3 + 0.3929 s^2 - 0.0064 s - 0.0111}{s^4 + 0.3846 s^3 + 0.6211 s^2 + 0.1852 s + 0.0186}$$
$$K_{21}=\frac{-0.5581 s^4 + 0.9026 s^3 - 0.2138 s^2 - 0.01 s - 0.03345}{s^4 + 0.3846 s^3 + 0.6211 s^2 + 0.1852 s + 0.0186}$$
$$K_{22}=\frac{0.2714 s^4 + 3.214 s^3 + 1.429 s^2 + 0.1429 s + 0.017}{s^4 + 0.3846 s^3 + 0.6211 s^2 + 0.1852 s + 0.0186}$$

Fig.~\ref{afterstabilization} displays the poles of $P_\lambda,\lambda\in[0,1]$ after stabilization. All feedback systems are stable because all of the poles are in the open unit disc.

\begin{figure}[htp]
	\centering
	\includegraphics[width=1\linewidth]{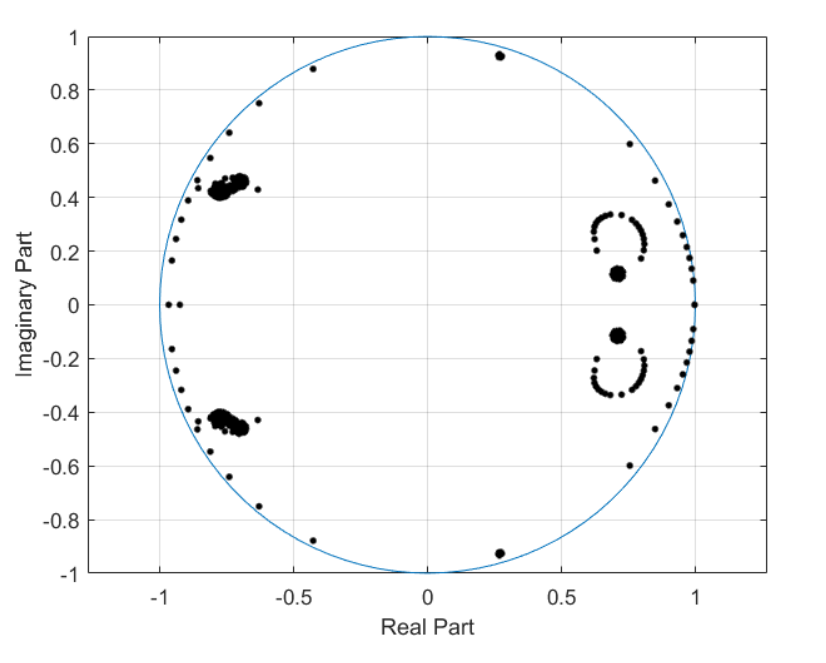}
	\caption{The poles of $P_\lambda$ after stabilization}
	\label{afterstabilization}
\end{figure}

\section{conclusion}\label{sec:conclusion}
In this study, we focus on the MIMO simultaneous stabilization issue and reframe it as an analytic interpolation problem. We resolve this problem by employing a Riccati-type algebraic matrix equation known as the Covariance Extension Equation. Furthermore, we present various solutions by selecting different matrix polynomials. In future research, we intend to incorporate derivative constraints.
\bibliographystyle{IEEEtran}

\end{document}